\documentclass[reqno]{amsart}

\usepackage[english]{babel}
\usepackage{amsthm, latexsym, enumerate, gastex, amssymb,url,listings}

\usepackage[mathcal]{eucal}
\usepackage{graphicx,url}

\copyrightinfo{2014}{Simone Ugolini}

\DeclareMathOperator{\Irr}{Irr}

\renewcommand{\phi}[0]{\varphi}
\renewcommand{\theta}[0]{\vartheta}
\renewcommand{\epsilon}[0]{\varepsilon}

\newcommand{\Z}{\text{$\mathbf{Z}$}}

\newcommand{\Pro}{\text{$\mathbf{P}^1$}}
\newcommand{\F}{\text{$\mathbf{F}$}}

\newcommand{\Mod}[1]{\  (\text{mod}\ #1)}

\newtheorem{theorem}{Theorem}[section]
\newtheorem{lemma}[theorem]{Lemma}

\theoremstyle{definition}

\newtheorem{definition}[theorem]{Definition}
\newtheorem{example}[theorem]{Example}

\theoremstyle{remark}
\newtheorem{remark}[theorem]{Remark}

\numberwithin{equation}{section}

\begin{document}

\bibliographystyle{amsplain}

\date{}

\title[]
{Sequences of irreducible polynomials over odd prime fields via elliptic curve endomorphisms}

\author{S.~Ugolini}
\email{sugolini@gmail.com} 

\begin{abstract}
In this paper we present and analyse a construction of irreducible polynomials over odd prime fields via the transforms which take any polynomial $f \in \F_p[x]$ of positive degree $n$ to $\left( \frac{x}{k} \right)^n \cdot f(k(x+x^{-1}))$, for some specific values of the odd prime $p$ and $k \in \F_p$.  
\end{abstract}

\maketitle

\section{Introduction}
Let $f$ be a polynomial of positive degree $n$ defined over the field $\F_p$ with $p$ elements, for some odd prime $p$. We set $q=p^n$ and denote by $\F_q$ the finite field with $q$ element. 

For a chosen $k \in \F_p^*$ we define the $Q_k$-transform of $f$ as
\begin{equation*}
f^{Q_k} (x) = \left( \frac{x}{k} \right)^n \cdot f(\theta_{k}(x)),
\end{equation*} 
where $\theta_{k}$ is the map which takes any element $x \in \Pro (\F_q)  = \F_q \cup \{ \infty \}$ to
\begin{displaymath}
\theta_{k} (x) = 
\begin{cases}
\infty & \text{if $x= 0$ or $\infty$,}\\
k \cdot (x+x^{-1}) & \text{otherwise}.
\end{cases}
\end{displaymath}

The aforementioned $Q_k$-transforms seem a natural generalization of some specific transforms employed by different authors for the synthesis of irreducible polynomials over finite fields. In \cite{mey} Meyn used the so-called $Q$-transform, which coincides with the $Q_1$-transform according to the notations of the present paper. Moreover, setting $k = \frac{1}{2}$ we recover the $R$-transform introduced by Cohen \cite{coh} and used more recently by us \cite{SUwpc} to construct sequences of irreducible polynomials over odd prime fields. 

In this paper we would like to take advantage of the knowledge of the dynamics of the maps $\theta_k$ for some specific values of $k$ \cite{SUk} and extend our investigation \cite{SUwpc}. In the following we will give a thorough description of the sequences of irreducible polynomials constructed by repeated applications of a $Q_k$-transform, when $k$ belongs to one of the following sets:  

\begin{itemize}
\item $C_1 = \left\{\frac{1}{2}, - \frac{1}{2} \right\}$;
\item $C_2 = \left\{ \text{$k \in \F_p: k$ is a root of $x^2 + \frac{1}{4}$} \right\}$, provided that $p \equiv 1 \pmod{4}$;
\item $C_3 = \left\{ \text{$k \in \F_p: k$  is a root of $x^2 + \frac{1}{2} x + \frac{1}{2}$} \right\}$, provided that $p \equiv 1, 2$, or $4 \pmod{7}$;
\item $C_3^- = \left\{ \text{$k \in \F_p: -k$ is a root of $x^2 + \frac{1}{2} x + \frac{1}{2}$} \right\}$, provided that $p \equiv 1, 2$, or $4 \pmod{7}$. 
\end{itemize}

Indeed, the case $k = \frac{1}{2}$ has been analysed in \cite{SUwpc} and we can easily adapt the results of that paper to the case $k = - \frac{1}{2}$ (see the subsequent Remark \ref{-1/2}). Hence, in this paper we will mainly concentrate on the cases that $k \in  C_2 \cup C_3 \cup C_3^-$. 

\section{Preliminaries}
Let $p$ be an odd prime and $q$ a power of $p$. For a fixed $k \in \F_p^*$, the dynamics of the map $\theta_k$ over $\Pro(\F_q)$ can be visualized by means of the graph $G^q_{\theta_{k}}$, whose vertices are labelled by the elements of $\Pro(\F_q)$ and where a vertex $\alpha$ is joined to a vertex $\beta$ if $\beta = \theta_k (\alpha)$. As in \cite{SUk} we say that an element $x \in \Pro (\F_q)$ is $\theta_{k}$-periodic if $\theta_{k}^r (x) = x$ for some positive integer $r$. We will call the smallest of such integers $r$ the period of $x$ with respect to the map $\theta_{k}$. Nonetheless, if an element $x \in \Pro (\F_q)$ is not $\theta_{k}$-periodic, then it is preperiodic, namely $\theta_{k}^l (x)$ is $\theta_{k}$-periodic for some positive integer $l$. 

In \cite{SUk} the  reader can find more details about the length and the number of the cycles of $G_{\theta_{k}}^q$, when $k \in C_1 \cup C_2 \cup C_3$. For the purposes of the present paper we are just interested in the structure of the reversed binary trees attached to the vertices of a cycle.

The following lemma shows how the maps $\theta_k$ and $\theta_{-k}$ are related, for any $k \in \F_p^*$.
\begin{lemma}\label{k-k}
Let $k \in \F_p^*$ and $x \in \Pro (\F_q)$. The following hold:
\begin{enumerate}
\item $\theta_k^{2r} (x) = \theta_{-k}^{2r} (x)$ for any nonnegative integer $r$;
\item if $\theta_k^t(x)$ is $\theta_k$-periodic, for some nonnegative integer $t$,  then $\theta_{-k}^t (x)$ is $\theta_{-k}$-periodic too.
\end{enumerate}
\end{lemma}
\begin{proof}
We prove separately the statements.
\begin{enumerate}[(1)]
\item We proceed by induction on $r$.

If $r=0$, then $\theta_{k}^0 (x) = \theta_{-k}^0 (x) = x$.

For the inductive step, assume that $\theta_{k}^{2(r-1)} (x) = \theta_{-k}^{2(r-1)} (x) = \tilde{x}$ for some integer $r > 0$. Therefore,
\begin{eqnarray*}
\theta_{k}^{2r} (x) & = & \theta_{k}^2 (\theta_k^{2r-2} (x)) = \theta_{k}^2 (\tilde{x}) =  k \cdot \frac{\left( k \frac{\tilde{x}^2+1}{\tilde{x}} \right)^2+1}{k \cdot \frac{\tilde{x}^2+1}{\tilde{x}}} \\
& = & -k \cdot \frac{\left( -k \frac{\tilde{x}^2+1}{\tilde{x}} \right)^2+1}{-k \cdot \frac{\tilde{x}^2+1}{\tilde{x}}} = \theta_{-k}^2 (\tilde{x}) = \theta_{-k}^{2r} (x).
\end{eqnarray*}

\item Let $\tilde{x} = \theta_k^t (x)$. By hypothesis, $\theta_k^r (\tilde{x}) = \tilde{x}$ for some nonnegative integer $r$. Then, $ \theta_{-k}^{2r} (\tilde{x}) = \theta_{k}^{2r} (\tilde{x}) = \tilde{x}$ according to (1).
\end{enumerate}
\end{proof}

\begin{remark}\label{-1/2}
In virtue of Lemma \ref{k-k}, the results in \cite[Theorem 3.1]{SUwpc} still hold replacing $\theta_{\frac{1}{2}}$ with $\theta_{-\frac{1}{2}}$ and the $R$-transform with the $Q_{-\frac{1}{2}}$-transform. 
\end{remark}

We introduce the following notations in analogy with \cite{SUwpc}.

\begin{definition}
If $f \in \F_p [x] \backslash \{ x \}$ is a monic irreducible polynomial and $\alpha$ is a non-zero root of $f$ in an appropriate extension of $\F_p$, then we denote by $\tilde{f}_{\theta_k}$ the minimal polynomial of $\theta_k (\alpha)$ over $\F_p$.
\end{definition}

\begin{definition}
We denote by $\Irr_p$ the set of all monic irreducible polynomials of $\F_p [x]$. If $n$ is a positive integer, then $\Irr_p(n)$ denotes the set of all polynomials of $\Irr_p$ of degree $n$.
\end{definition}

\begin{remark}
The reader can notice a slight difference in the definition of $\Irr_p$ with respect to \cite[Definition 2.3]{SUwpc}. Indeed, in \cite{SUwpc} we excluded the polynomials $x+1$ and $x-1$ from $\Irr_p$, because $1$ and $-1$ are the only $\theta_{\frac{1}{2}}$-periodic elements in $G^{q}_{\frac{1}{2}}$ which are not root of any reversed binary tree, for any power $q$ of $p$. If $k \in C_2 \cup C_3 \cup C_3^-$ this phenomenon does not occur, namely any $\theta_k$-periodic element is root of a reversed binary tree. 
\end{remark}

The following lemma and theorems can be proved respectively as \cite[Lemma 2.5, Theorem 2.6, Theorem 2.7]{SUwpc} using the current more general notations of $\theta_k$ and $f^{Q_k}$ in place of $\theta_{\frac{1}{2}}$ and $f^R$.
\begin{lemma}\label{seq_2}
Let $f$ be a polynomial of positive degree $n$ in $\F_p [x]$. Suppose that $\beta$ is a root of $f$ and that $\beta = \theta_{k} (\alpha)$ for some $\alpha, \beta$ in suitable extensions of $\F_p$. Then, $\alpha$ and $\alpha^{-1}$ are roots of $f^{Q_k}$.
\end{lemma}

\begin{theorem}
Let $f$ be a polynomial of $\Irr_p(n) \backslash \{ x, x+1, x-1 \}$, for some positive integer $n$. The following hold.
\begin{itemize}
\item If the set of roots of $f$ is not inverse-closed, then $\tilde{f}_{\theta_k} \in \Irr_p(n)$.
\item If the set of roots of $f$ is inverse-closed, then $n$ is even and $\tilde{f}_{\theta_k} \in \Irr_p(n/2)$.
\end{itemize}
\end{theorem}

\begin{theorem}\label{seq_4}
Let $f(x) = x^n+a_{n-1} x^{n-1} + \dots + a_1 x + a_0 \in \Irr_p (n)$ for some positive integer $n$. The following hold.
\begin{itemize}
\item $0$ is not a root of $f^{Q_k}$.
\item The set of roots of $f^{Q_k}$ is closed under inversion.
\item Either $f^{Q_k} \in \Irr_p (2n)$ or $f^{Q_k}$ splits into the product of two polynomials $m_{\alpha}, m_{\alpha^{-1}}$ in  $\Irr_p (n)$, which are respectively the minimal polynomial of $\alpha$ and $\alpha^{-1}$, for some $\alpha \in \F_{p^n}$. Moreover, in the latter case at least one among $\alpha$ and $\alpha^{-1}$ is not $\theta_{k}$-periodic.
\end{itemize}
\end{theorem}

Before dealing with the construction of sequences of irreducible polynomials, we prove some additional results for the dynamics of the maps $\theta_k$, when $k$ belongs respectively to $C_2$, $C_3$ or $C_3^-$. Such results complement our investigation \cite{SUk}.

\subsection{Case $k \in C_2$.}\label{case_2}
In this section we fix an odd prime $p \equiv 1 \pmod{4}$ and a root $k$ of $x^2 + \frac{1}{4}$ in $\F_p$. 

The notations of the current section are the same as in \cite[Section 3]{SUk}, except for the symbol $\nu_2$, which we introduce in the present paper. For the reader's convenience we review all the relevant notations:
\begin{itemize}
\item $R = \Z[i]$;
\item $\alpha = 1 + \sqrt{-1} \in R$;
\item $N$ is the norm function on $R$, which takes any element $a + ib \in \Z[i]$ to $N(a+ib) = a^2+b^2$;
\item $E$ is the elliptic curve of equation $y^2 = x^3+x$ over $\F_p$;
\item $\pi_p$ is the representation in $R$ of the Frobenius endomorphism of $E$ over $\F_p$, namely the map which takes any point $(x,y)$ of $E$ to $(x^p, y^p)$;
\item $E(\F_{p^n})$ is the group of rational points of $E$ over $\F_{p^n}$, for any positive integer $n$, while
\begin{eqnarray*}
E(\F_{p})^*  & = &  \{  O , (0,0), (i_p, 0), (-i_p, 0) \} \subseteq E(\F_{p^n}),\\
E(\F_{p})^*_x & =  & \{  \infty , 0, i_p, - i_p \} \subseteq \Pro(\F_{p^n}),
\end{eqnarray*}
being $O$ the point at infinity of  $E$ and $i_p$ a square root of $-1$ in $\F_p$;
\item $\rho_0$ is the element of $R$ defined as
\begin{displaymath}
\rho_0 = 
\begin{cases}
\alpha, & \text{if $\alpha^{-2} \equiv k \Mod{\pi_p}$,}\\
\overline{\alpha}, & \text{if $\overline{\alpha}^{-2} \equiv k \Mod{\pi_p}$};
\end{cases}
\end{displaymath}
\item $\Pro(\F_{p^n}) = A_n \cup B_n$, where $A_n$ and $B_n$ are two disjoint subsets of $\Pro(\F_{p^n})$ for any positive integer $n$;
\item for any positive integer $m$, which we can express as $m =2^e \cdot f$ for some odd integer $f$ and non-negative integer $e$,  we denote by $\nu_2 (m)$ the exponent of the greatest power of $2$ which divides $m$, namely $\nu_2 (m) = e$.
\end{itemize}

According to \cite[Lemma 3.2]{SUk}, either all the elements belonging to a connected component of $G^{p^n}_{\theta_k}$ are in $A_n$ or they are in $B_n$. 

In \cite[Section 3]{SUk} we defined the sets
\begin{eqnarray*}
E(\F_{p^{n}})_{A_n} & = & \left\{(x,y) \in E(\F_{p^{n}}) : x \in A_n \backslash \{ \infty \} \right\},\\
E(\F_{p^{2n}})_{B_n} & = & \left\{(x,y) \in E(\F_{p^{2n}}) : x \in B_n \backslash \{ \infty \} \right\}
\end{eqnarray*} 
and proved the existence of two isomorphisms
\begin{displaymath}
\begin{array}{lrcl}
\psi_n : & E(\F_{p^{n}})_{A_n} \cup E(\F_p)^* & \to & R / (\pi_p^n-1) R,\\
\widetilde{\psi}_n : & E(\F_{p^{2n}})_{B_n} \cup E(\F_p)^* & \to & R / (\pi_p^n+1) R.
\end{array}
\end{displaymath}

The dynamics of the map $\theta_k$ on $A_n$ (resp. $B_n$) can be studied relying upon the iterations of $[\rho_0]$ in $R / (\pi_p^n-1) R$ (resp. $R / (\pi_p^n+1) R$). In particular, according to \cite[Theorem 3.5]{SUk}, the depth of the trees in a connected component formed by elements of $A_n$ (resp. $B_n$) is equal to $e_0$, being $\rho_0^{e_0}$ the greatest power of $\rho_0$ which divides $\pi_p^n-1$ (resp. $\pi_p^n+1$). Indeed, since $N(\rho_0)=2$ and the norm of any  irreducible factor of $\pi_p^n-1$ (resp. $\pi_p^n+1$) different from $\rho_0$ and $\overline{\rho}_0$ is odd, we have that $e_0 = \nu_{2} (N(\pi_p^n-1))$ (resp. $\nu_{2} (N(\pi_p^n+1))$).

We prove the following technical lemma.

\begin{lemma}\label{seq_3}
Let $m$ and $n$ be two positive integers. Then, the following hold:
\begin{enumerate}
\item $\nu_2 (N(\pi_p^n - 1))  \geq 2$;
\item if $\nu_2 (N(\pi_p^n - 1)) = 2$, then $\nu_2 (N(\pi_p^n + 1))  \geq 3$;
\item if $\nu_2 (N(\pi_p^n - 1)) \geq 3$, then $\nu_2 (N(\pi_p^n + 1)) = 2$;
\item $\nu_2 (N(\pi_p^{2^{i+1} m}-1)) = \nu_2 (N(\pi_p^{2^{i} m}-1))+2$, for any positive integer $i$.
\end{enumerate}
\end{lemma}
\begin{proof}
Since $\pi_p^n-1 \in \Z[i]$, we have that $\pi_p^n-1 = a+ib$, for some $a, b \in \Z$, and consequently $\pi_p^n+1 = (a+2)+ib$. We prove separately the statements.
\begin{enumerate}[(1)]
\item Let
\begin{equation*}
S = R / \rho_0^{e_0} R \times R / \rho_1^{e_1} R \cong R / (\pi_p^n-1) R
\end{equation*} 
for some element $\rho_1 \in R$ coprime to $\rho_0$ and some non-negative integers $e_0, e_1$.

Consider the following points in $S$:
\begin{eqnarray*}
Q & = & (Q_0, Q_1) = \psi_n (i_p, 0);\\
P & = & (P_0, P_1) = \psi_n (0,0);\\
\mathcal{O} & = & (0,0) = \psi_n (O).
\end{eqnarray*}
Since $\theta_k (i_p) = 0$ and $\theta_k (0) = \infty$, we have that $[\rho_0] Q = P$ and $[\rho_0] P = \mathcal{O}$. Moreover, by the fact that $\rho_0$ and $\rho_1$ are coprime, we deduce that $P_1 = 0$ and $Q_1 = 0$. In addition, $Q_0 \not = 0$ and $P_0 \not = 0$. Therefore, $[\rho_0] Q_0 \not = 0$ in $R / \rho_0^{e_0} R$. Since this latter is true only if $e_0 \geq 2$, we get the result. 
\item By hypothesis, 
\begin{equation}\label{intro_eq_1}
N(\pi_p^n-1) = a^2 + b^2 = 4 c,
\end{equation}
for some odd integer $c$, and consequently $a$ and $b$ have the same parity. Indeed, $a$ and $b$ are both even. Suppose, on the contrary, that $a$ and $b$ are both odd. Then, $a \equiv \pm 1 \pmod{4}$, $b \equiv \pm 1 \pmod{4}$ and $a^2+b^2 \equiv 2 \pmod{4}$, in contradiction with (\ref{intro_eq_1}).

Evaluating $N(\pi_p^n+1)$ we get 
\begin{equation*}
N(\pi_p^n+1) = a^2 + 4 a + 4 + b^2 = 4 c + 4 (1+a) = 4 (1+a+c).
\end{equation*} 
We notice that $1+a$ is odd, because $a$ is even. Therefore, $1+a+c$ is even and consequently $N(\pi_p^n+1) \equiv 0 \pmod{8}$. Hence, $\nu_2 (N(\pi_p^n + 1))  \geq 3$.

\item By hypothesis, 
\begin{equation}
N(\pi_p^n-1) = a^2 + b^2 = 8 c,
\end{equation}
for some integer $c$. In particular, $a$ and $b$ are both even, as proved in (2). 

We evaluate $N(\pi_p^n+1)$ and get
\begin{equation*}
N(\pi_p^n+1) = a^2 + 4 a + 4 + b^2 = 8 c + 4 (1+a).
\end{equation*} 

We notice that $1+a$ is odd, because $a$ is even. Therefore,
\begin{equation*}
N(\pi_p^n+1) \equiv 4 \cdot (1+a) \not \equiv 0 \pmod{8}
\end{equation*}
and consequently $\nu_2 (N(\pi_p^n + 1)) = 2$.

\item Set $n:= 2^{i-1} m$. Then,
\begin{eqnarray*}
\nu_2 (N(\pi_p^{2^{i} m}-1)) & = & \nu_2 (N(\pi_p^{2 n}-1)) =  \nu_2 ( N(\pi_p^{n}-1) ) + \nu_2 ( N(\pi_p^{n}+1) ).
\end{eqnarray*}
From (1), (2) and (3) we get that $\nu_2 ( N(\pi_p^{n}-1) ) + \nu_2 ( N(\pi_p^{n}+1) ) \geq 5$. Hence, $\nu_2 (N(\pi_p^{2^{i} m}-1)) \geq 5$.

Set now $n:=2^i m$. Then,
\begin{eqnarray*}
\nu_2 (N(\pi_p^{2^{i+1} m}-1)) & = & \nu_2 ( N(\pi_p^{n}-1) ) + \nu_2 ( N(\pi_p^{n}+1) ).
\end{eqnarray*} 
Since $\nu_2 (N(\pi_p^{2^{i} m}-1)) \geq 5$, from (3) we get that $\nu_2 ( N(\pi_p^{n}+1) ) =2$. All considered, 
\begin{eqnarray*}
\nu_2 (N(\pi_p^{2^{i+1} m}-1)) & = &  \nu_2 (N(\pi_p^{2^{i} m}-1)) + 2.
\end{eqnarray*} 
\end{enumerate}
\end{proof}

\subsection{Case $k \in C_3$}\label{case_3}
In this section we fix an odd prime $p \equiv 1, 2,$ or $4 \pmod{7}$ and a root $k$ of $x^2 + \frac{1}{2} x + \frac{1}{2}$ in $\F_p$.

The notations of the current section are the same as in \cite[Section 4]{SUk}, except for the symbol $\nu_{r}$, which we introduce in the present paper. For the reader's convenience we review all the relevant notations:
\begin{itemize}
\item $R = \Z[\alpha]$, where $\alpha = \frac{1+\sqrt{-7}}{2}$;
\item $N$ is the norm function on $R$, which takes any element $a + b \alpha \in \Z[\alpha]$ to $(a+b \alpha) \cdot \overline{(a+b \alpha)}$;
\item $E$ is the elliptic curve of equation $y^2 = x^3 - 35 x + 98$ over $\F_p$;
\item $\pi_p$ is the representation in $R$ of the Frobenius endomorphism of $E$ over $\F_p$;
\item $\sigma \equiv 2 k +1 \pmod{p}$, while $\overline{\sigma} \equiv - 2 k \pmod{p}$;
\item $E(\F_{p^n})$ is the group of rational points of $E$ over $\F_{p^n}$, for any positive integer $n$, while
\begin{eqnarray*}
E(\F_{p})^*  & = & \{  O , (-7,0), (\sigma+3, 0), (\overline{\sigma}+3, 0) \} \subseteq E(\F_{p^n}),\\
E(\F_{p})^*_x & = &  \{  \infty , -7, \sigma+3, \overline{\sigma}+3 \} \subseteq \Pro(\F_{p^n}),
\end{eqnarray*}
being $O$ the point at infinity of $E$;
\item $\rho_0$ is the element of $R$ defined as 
\begin{displaymath}
\rho_0 = 
\begin{cases}
\alpha, & \text{if $\alpha \equiv \sigma \Mod{\pi_p}$,}\\
\overline{\alpha}, & \text{if $\overline{\alpha} \equiv \sigma \Mod{\pi_p}$};
\end{cases}
\end{displaymath}
\item $\Pro(\F_{p^n}) = A_n \cup B_n$, where $A_n$ and $B_n$ are two disjoint subsets of $\Pro(\F_{p^n})$ for any positive integer $n$;
\item for any $r \in R$ such that $r = r_0^{e_0} \cdot r_1$,  where $r_0$ and $r_1$ are two coprime elements of $R$ and $e_0$ is a nonnegative integer, we denote by $\nu_{r_0} (r)$ the exponent of the greatest power of $r_0$ which divides $r$, namely $\nu_{r_0} (r) = e_0$.
\end{itemize}

In \cite[Section 4]{SUk} we introduced the rational maps $\eta_k$ and $\chi_k$, defined on $\Pro(\F_{p^n})$ in such a way that
\begin{equation*}
\theta_k (x) = \chi_k^{-1} \circ \eta_k \circ \chi_k(x)
\end{equation*}
for any $x \in \Pro (\F_{p^n})$. In virtue of this fact, the graphs $G^{p^n}_{\theta_k}$ and $G^{p^n}_{\eta_k}$ are isomorphic. 

According to \cite[Lemma 4.4]{SUk}, either all the elements belonging to a connected component of $G^{p^n}_{\eta_k}$ are in $A_n$ or they are in $B_n$. As in Section \ref{case_2} we define the sets $E(\F_{p^{n}})_{A_n}$ and $E(\F_{p^{2n}})_{B_n}$ and the isomorphisms $\psi_n$ and $\widetilde{\psi}_n$.

The dynamics of the map $\eta_k$ on $A_n$ (resp. $B_n$) can be studied relying upon the iterations of $[\rho_0]$ in $R / (\pi_p^n-1) R$ (resp. $R / (\pi_p^n+1) R$). In particular, according to \cite[Theorem 4.6]{SUk}, the depth of the trees in a connected component formed by elements of $A_n$ (resp. $B_n$) is equal to $e_0$, being $\rho_0^{e_0}$ the greatest power of $\rho_0$ which divides $\pi_p^n-1$ (resp. $\pi_p^n+1$), namely $e_0$ is equal to $\nu_{\rho_0} (\pi_p^n-1)$ (resp. $\nu_{\rho_0} (\pi_p^n+1)$).

We prove a technical lemma, which provides some useful results for the construction of the sequences of irreducible polynomials. 

\begin{lemma}\label{tech_lemma_case_3}
Let $n$ be a positive integer. Then, the following hold:
\begin{enumerate}
\item $\nu_{\rho_0} (\pi_p^n-1) \geq 1$;
\item if $\nu_{\rho_0} (\pi_p^n-1) = 1$, then $\nu_{\rho_0} (\pi_p^n+1) \geq 2$;
\item if $\nu_{\rho_0} (\pi_p^n-1) \geq 2$, then $\nu_{\rho_0} (\pi_p^n+1) = 1$;
\item $\nu_{\rho_0} (\pi_p^{2^{i+1}n}-1) = \nu_{\rho_0} (\pi_p^{2^i n} -1) + 1$, for any positive integer $i$.
\end{enumerate}
\end{lemma}
\begin{proof}
Let
\begin{eqnarray*}
S = R / \rho_0^{e_0} R \times R / \rho_1^{e_1} R & \cong & R / (\pi_p^n-1) R,\\
\tilde{S} = R / \rho_0^{\tilde{e}_0} R \times R / \tilde{\rho}_1^{\tilde{e}_1} R & \cong & R / (\pi_p^n+1) R,
\end{eqnarray*}
for some elements $\rho_1$ and $\tilde{\rho}_1$ in $R$ coprime to $\rho_0$ and some nonnegative integers $e_0, e_1, \tilde{e}_0$ and $\tilde{e}_1$. Moreover, define $x_P = \overline{\sigma}+3$ and $x_Q = 2 \sigma -1$. Then, denote by $y_Q$ and $-y_Q$ the $y$-coordinates of the two rational points  $E(\F_{p^{2n}})$ having $x_Q$ as $x$-coordinate. We remind the reader that, according to \cite[Lemma 4.3]{SUk},
\begin{eqnarray*}
\eta_k(x_Q) & = & x_P;\\
\eta_k(x_P) & = & \infty.
\end{eqnarray*}

Consider the points $\mathcal{O}, P$ in $S$ and the points $\tilde{\mathcal{O}}, \tilde{P}$ in $\tilde{S}$ defined as follows:
\begin{eqnarray*}
\mathcal{O} & = & (0,0) = \psi_n (O);\\
\tilde{\mathcal{O}} & = & (0,0) = \widetilde{\psi}_n (O);\\
P & = & (P_0, P_1) = \psi_n (x_P, 0);\\
\tilde{P} & = & (\tilde{P}_0, \tilde{P}_1) = \widetilde{\psi}_n (x_P, 0).
\end{eqnarray*}
\begin{enumerate}[(1)]
\item Since $\eta_k(x_P) = \infty$,  we have that  $[\rho_0] P = ([\rho_0] P_0, [\rho_0] P_1) = \mathcal{O}$. Indeed, $P_1 = 0$, because $\rho_0$ and $\rho_1$ are coprime. Moreover, $P \not = \mathcal{O}$ and consequently $P_0 \not = 0$.  Therefore, $e_0 = \nu_{\rho_0} (\pi_p^n-1) \geq 1$. 
\item Since $x_Q \in \F_{p^n}$, either $x_Q$ belongs to $A_n$ or $x_Q$ belongs to $B_n$. 

Suppose that $x_Q \in A_n$ and define $Q  = (Q_0, Q_1) = \psi_n (x_Q, y_Q) $. Then, $[\rho_0] Q = P$ and $[\rho_0]^2 Q = \mathcal{O}$. In particular, $Q_1 = 0$, because $\rho_0$ and $\rho_1$ are coprime. Moreover, since $e_0 =1$ by hypothesis, $[\rho_0] Q_0 = 0$ and consequently $P_0=0$. This latter is absurd and we deduce that $x_Q \in B_n$. 

Define now $\tilde{Q}  = (\tilde{Q}_0, \tilde{Q}_1) = \widetilde{\psi}_n (x_Q, y_Q)$. Since $[\rho_0] \tilde{Q} = \tilde{P} \not = \mathcal{\tilde{O}}$, we conclude that $\tilde{e}_0 = \nu_{\rho_0} (\pi_p^n+1) \geq 2$.

\item Since $\tilde{P} \in \tilde{S}$ and $\tilde{P} \not = \tilde{\mathcal{O}}$, we deduce that $\nu_{\rho_0} (\pi_p^n+1) \geq 1$. Indeed, $\nu_{\rho_0} (\pi_p^n+1) = 1$. Suppose, on the converse, that $\nu_{\rho_0} (\pi_p^n+1) \geq 2$. Consider the point $\tilde{R} = ([\rho_0]^{\tilde{e}_0-2}, 0) \in \tilde{S}$. Since $[\rho_0]^2 \tilde{R} = \tilde{\mathcal{O}}$, we have that $[\rho_0] \tilde{R} = \tilde{P}$. In particular, $\tilde{R} \in \left\{\widetilde{\psi}_n (x_Q, y_Q), \widetilde{\psi}_n (x_Q, - y_Q) \right\}$. As a consequence, $x_Q \in B_n$. Since by hypothesis $e_0 \geq 2$, any tree rooted in an element of $A_n$ has depth at least $2$. Therefore, $x_Q \in A_n$, because it belongs to the level $2$ of the tree of $G^{p^n}_{\eta_k}$ rooted in $\infty$. All considered, we get a contradiction due to the fact that $A_n \cap B_n = \emptyset$ by definition.

\item From (1), (2) and (3) we deduce that
\begin{equation*}
\nu_0 (\pi_p^{2^i n}-1) = \nu_0 (\pi_p^{2^{i-1} n}-1) + \nu_0 (\pi_p^{2^{i-1} n}+1) \geq 3.
\end{equation*}
Finally, according to (3),
\begin{equation*}
\nu_0 (\pi_p^{2^{i+1} n}-1) = \nu_0 (\pi_p^{2^{i} n}-1) + \nu_0 (\pi_p^{2^{i} n}+1) = \nu_0 (\pi_p^{2^{i} n}-1) +1.
\end{equation*}
\end{enumerate}
\end{proof}

\subsection{Case $k \in C_3^-$}\label{case_3-}
According to Lemma \ref{k-k}, if $k \in C_3^-$ and $n$ is a positive integer, then the dynamics of $\theta_k$ on $\Pro(\F_{p^n})$ is strictly related to the dynamics of $\theta_{-k}$, a map which belongs to the family of maps investigated in \cite[Section 4]{SUk}. Indeed, an element $\tilde{x} \in \Pro(\F_{p^n})$ is $\theta_k$-periodic if and only if it is $\theta_{-k}$-periodic. In the case that $\tilde{x}$ is not $\theta_k$-periodic, $\tilde{x}$ belongs to a certain level $t$ of some tree both in $G^{p^n}_{\theta_k}$ and in $G^{p^n}_{\theta_{-k}}$.

\section{Constructing  irreducible polynomials via the $Q_k$-transforms}
The following theorem describes how the iterative procedure for constructing irreducible polynomials via the $Q_k$-transforms works, when $k \in C_2 \cup C_3 \cup C_3^-$. 
\begin{theorem}\label{thm_seq_case_2}
Let $f_0 \in \Irr_p(n)$, for some odd prime $p$ and some positive integer $n$, and $k \in C_2 \cup C_3 \cup C_3^-$. 

Define two nonnegative integers $e_0$ and $e_1$ as follows:
\begin{itemize}
\item if $k \in C_2$, then
\begin{eqnarray*}
e_0 & = & \nu_2 (N(\pi_p^n-1)),\\
e_1 & = & \nu_2 (N(\pi_p^n+1)),
\end{eqnarray*}
where $\nu_2, N$ and $\pi_p$ are defined as in Section \ref{case_2};
\item if $k \in C_3 \cup C_3^-$, then
\begin{eqnarray*}
e_0 & = & \nu_{\rho_0} (\pi_p^n-1),\\
e_1 & = & \nu_{\rho_0} (\pi_p^n+1),
\end{eqnarray*}
where $\rho_0, \nu_{\rho_0}$ and $\pi_p$ are defined as in Section \ref{case_3}.
\end{itemize}

If $f_0^{Q_k}$ is irreducible, define $f_1 := f_0^{Q_k}$. Otherwise, as stated in Theorem \ref{seq_4}, factor $f_0^{Q_k}$ into the product of two monic irreducible polynomials $g_1, g_2$ of the same degree $n$, where $g_1$ has a  non-$\theta_k$-periodic root in $\F_{p^n}$. In this latter case we set $f_1:=g_1$.

For $i \geq 2$ define inductively a sequence of polynomials $\{ f_i \}_{i \geq 2}$ in such a way:
\begin{itemize}
\item if $f_{i-1}^{Q_k}$ is irreducible, then set $f_{i} := f_{i-1}^{Q_k}$;
\item if $f_{i-1}^{Q_k}$ is not irreducible, then factor $f_{i-1}^{Q_k}$ into the product of two monic irreducible polynomials $g_1, g_2$ of the same degree and set $f_{i} := g_1$.
\end{itemize} 
Then, there exist a nonnegative integer $s$ and a positive integer $t$ such that:
\begin{itemize}
\item $s \leq \max \{e_0, e_1 \}$;
\item $s+t \leq e_0 + e_1$;
\item $\{f_0, \dots, f_{s} \} \subseteq \Irr_p(n)$;
\item $\{f_{s+1}, \dots, f_{s+t} \} \subseteq \Irr_p(2n)$.
\end{itemize}
Moreover, the following hold:
\begin{itemize}
\item if $k \in C_2$, then $\{ f_{s+t+2j-1}, f_{s+t+2j} \} \subseteq \Irr_p \left( 2^{j+1} \cdot n \right)$ for any $j \geq 1$;
\item if $k \in C_3 \cup C_3^-$, then $f_{s+t+j} \in \Irr_p \left( 2^{j+1} \cdot n \right)$ for any $j \geq 1$.
\end{itemize}
\end{theorem}
\begin{proof}
The proof of the present theorem follows the same lines of the proof of \cite[Theorem 3.1]{SUwpc}.

First, we denote by $\beta_0 \in \F_{p^n}$ a root of $f_0$. Then, we construct inductively a sequence $\{ \beta_i \}_{i \geq 0}$ of elements belonging to $\F_{p^n}$ or to appropriate extensions of $\F_{p^n}$ such that, for any $i \geq 0$, the following hold:
\begin{itemize}
\item $f_{i} (\beta_i) = 0$;
\item $\theta_k (\beta_{i+1}) = \beta_i$.
\end{itemize}

We notice that, being the roots of $f_1$ not $\theta_k$-periodic, $\beta_0$ is a vertex of some tree in $G_{\theta_k}^{p^n}$. Since the depth of a tree in $G_{\theta_k}^{p^n}$ is either equal to $e_0$ or to $e_1$, we conclude that there exists a non-negative integer $s \leq \max \{e_0, e_1 \}$ such that $\beta_s$ has degree $n$ over $\F_p$, while $\beta_{s+1}$ has degree $2n$ over $\F_p$, implying that $\{f_0, \dots, f_s \} \subseteq \Irr_p(n)$, while $f_{s+1} \in \Irr_p(2n)$.

We notice that $\beta_0$ is the root of a tree in $G_{\theta_k}^{p^{2n}}$ which has depth $e_0 + e_1$. Therefore, there exists a positive integer $t$, with $s+t \leq e_0 + e_1$, such that $\beta_{s+t}$ has degree $2n$ over $\F_p$, while $\beta_{s+t+1}$ has degree $4n$ over $\F_p$, implying that $\{f_{s+1}, \dots, f_{s+t} \} \subseteq \Irr_p(2n)$.

The last two statements regarding the polynomials $f_i$, for $i > s+t$, follow respectively from Lemma \ref{seq_3} and Lemma \ref{tech_lemma_case_3}.
\end{proof}

\begin{remark}
One of the hypotheses of Theorem \ref{thm_seq_case_2} is that the polynomial $f_1$ has no $\theta_k$-periodic roots. While this is true if $f_0^{Q_k} \in \Irr_p (2n)$, the same does not always hold if $f_0^{Q_k} (x) = g_1(x) \cdot g_2(x)$, for some monic irreducible polynomials $g_1, g_2$ of equal degree $n$. More precisely, at least one of $g_1$ and $g_2$ has no $\theta_k$-periodic roots. If one of them, say $g_1$, has $\theta_k$-periodic roots and we set $f:=g_1$, it can happen that $f_{\tilde{e}} \in \Irr_p (n)$, where $\tilde{e} = \max \{ e_0, e_1 \}$. If this is the case, then we break the iterative procedure and set $f_1 := g_2$. Doing that, the hypotheses of the theorem are satisfied and we can proceed with the iterative construction.  
\end{remark}

\begin{example}
Consider the prime $p=53$. We notice that $p \equiv 1 \pmod{4}$ and $p \equiv 4 \pmod{7}$. 

First, we fix a root $k$ of $x^2 + \frac{1}{4}$ in $\F_p$, namely $k = 15$, and construct a sequence of monic irreducible polynomials from the polynomial $f_0 (x) = x^5+3x+51 \in \Irr_{53} (5)$ via the transform $Q_{15}$. Proceeding as explained in Theorem \ref{thm_seq_case_2}, using a computational tool as \cite{GAP}, we get that $f_1, f_2$ and $f_3$ belong to $\Irr_{53} (10)$, while $f_4 \in \Irr_{53} (20)$. Therefore, in accordance with the notations and the claims of Theorem \ref{thm_seq_case_2}, in this example $s=1$, $t=2$ and $\{f_{3+2j-1}, f_{3+2j} \} \subseteq \Irr_{53} (2^{j+1} \cdot 5)$ for any $j \geq 1$.

We notice in passing that, while $f_{3+2j-1} = f_{3+2j-2}^{Q_{15}}$ for any $j \geq 1$, any polynomial $f_{3+2j}$ is equal to one of the two irreducible factors of $f_{3+2j-1}^{Q_{15}}$. This latter is equivalent to saying that every two steps in our construction the factorization of a polynomial is required. While efficient algorithms for the factorization of a polynomial into two equal-degree polynomials are known, one can reduce this burden taking $k \in C_3 \cup C_3^-$.

Consider now $k = 7 \in C_3$ and $f_0$ as before. Constructing a sequence of monic irreducible polynomials via the transform $Q_{7}$ starting from $f_0$ we get that $f_1 \in \Irr_{53} (10)$, while $f_2 \in \Irr_{53} (20)$. Therefore, in accordance with Theorem \ref{thm_seq_case_2}, in this example $s=0$, $t=1$ and $f_{1+j} \in \Irr_{53} (2^{j+1} \cdot 5)$ for any $j \geq 1$. In particular, $f_{j+1} = f_j^{Q_{7}}$ for $j \geq 1$ and no factorization is required. 
\end{example}

\bibliography{Refs}
\end{document}